\documentclass[12pt]{amsart}
\usepackage{amsmath, amssymb, amsfonts, amsthm}
\usepackage[utf8]{inputenc}
\usepackage[T1]{fontenc}
\usepackage{lmodern}
\usepackage{geometry}
\usepackage{xcolor}
\newtheorem{theorem}{Theorem}

\newtheorem{lemma}[theorem]{Lemma}

\newtheorem{proposition}[theorem]{Proposition}

\newtheorem*{question}{Question}

\title{Minimality of $T_z^*$-Invariant Subspaces of $H^2(\mathbb D^2)$}
\author{João Marcos R. do Carmo}\thanks{Instituto Federal da Bahia-Campus Seabra, Brazil. E-mail: joaoribeiro@ifba.edu.br (João Marcos R. do Carmo)}
\author{Marcos S. Ferreira}\thanks{Universidade Estadual de Santa Cruz, Brazil. E-mail: msferreira@uesc.br (Marcos S. Ferreira)}

\begin{document}

\begin{abstract}
In this paper, we investigate the minimality of infinite-dimensional $T_z^*$-invariant subspaces of $H^2(\mathbb{D}^2)$. Using model spaces, inner factorizations and cyclicity, we show that every nonzero coordinate projection of an infinite-dimensional minimal $T_z^*$-invariant subspace is dense in the
corresponding one-variable Hardy space. Since $T_z^*$ is a universal operator on $H^2(\mathbb{D}^2)$, our result reduces the Invariant Subspace Problem to the
question of whether an infinite-dimensional $T_z^*$-invariant subspace whose coordinate projections are either zero or dense can be minimal.
\end{abstract}

\maketitle

\noindent\textbf{Mathematics Subject Classification (2020):} 30H10, 47B35, 47A15.

\noindent\textbf{Keywords:} Invariant subspace problem, Toeplitz operator, Hardy space, coordinate projections.

\section{Introduction}
Let $T$ be a bounded operator on a Hilbert space $\mathcal{H}$. A closed subspace $M\subset\mathcal{H}$ is said to be $T$-invariant if $TM\subset M$. The Invariant Subspace Problem (ISP) for complex separable Hilbert spaces asks whether every bounded operator has a nontrivial invariant subspace. For an overview of some modern approaches to the ISP, we refer the reader to the recent monograph \cite{Chalendar}.

In 1960, Rota \cite{Rota} introduced the notion of a universal operator, whose invariant subspaces encode the action of arbitrary bounded operators. More precisely, a bounded operator $U$ on $\mathcal{H}$ is called universal if, for every bounded operator $A$ on $\mathcal{H}$, there exist a $U$-invariant subspace $M$, a nonzero scalar $\lambda$, and a linear isomorphism $X\colon\mathcal H\to M$ such that
$$
UX=\lambda XA.
$$

The connection between universal operators and the ISP has continued to attract attention; see, for instance, \cite{Cowen2,Cowen3,Carmo3}.

The existence of universal operators provides a useful reformulation of the ISP in terms of minimal invariant subspaces. In particular, for a universal operator $U$ on a separable, infinite-dimensional Hilbert space $\mathcal H$, the ISP is equivalent to the assertion that every infinite-dimensional $U$-invariant subspace is non-minimal.

Let $H^2(\mathbb{D}^2)$ denote the Hardy space over the bidisk. In \cite{Ferreira}, the authors used Caradus's criterion for universality to show that $T_z^*$ is universal on $H^2(\mathbb{D}^2)$. In this paper, we investigate the minimality of $T_z^*$-invariant subspaces of $H^2(\mathbb{D}^2)$. We first establish structural restrictions on the coordinate projections of such subspaces using model spaces and inner functions. By combining inner factorizations with the cyclicity arising from minimality, we then establish our main result (Theorem \ref{t6}): if $M\subset H^2(\mathbb{D}^2)$ is an infinite-dimensional minimal $T_z^*$-invariant subspace, then, for every $n\in\mathbb{N}_0$,
$$
P_n(M)=\{0\}
\quad\text{or}\quad
\overline{P_n(M)}=H^2(\mathbb{D}),
$$
where $P_n$ denotes the orthogonal projection onto $H^2(z)w^n$, identified with $H^2(\mathbb{D})$.

In particular, Theorem \ref{t6} rules out the atomic singular inner function alternative in Theorem \ref{t5}. In view of the universality of $T_z^*$, the Invariant Subspace Problem is thus equivalent to determining whether every infinite-dimensional $T_z^*$-invariant subspace whose coordinate projections are either zero or dense must be non-minimal.

\section{Preliminaries}

\subsection{The Hardy space $H^{2}(\mathbb{D}^{2})$}
The Hardy space $H^{2}(\mathbb{D}^2)$ consists of all holomorphic functions $f$ on the bidisk $\mathbb{D}^2$ for which
$$
\|f\|_{H^2(\mathbb{D}^2)}^2=\sup_{0<r<1}
\int_{\mathbb{T}^2}
|f(r\zeta_1,r\zeta_2)|^2\,d\sigma(\zeta)
<\infty
$$
where $d\sigma$ denotes normalized Lebesgue measure on $\mathbb{T}^{2}$.

The Toeplitz operator $T_{z}$ with symbol $z$ acts on $H^{2}(\mathbb{D}^2)$ as multiplication by the first coordinate function, that is,
$$
T_z f(z,w)=zf(z,w),
\qquad f\in H^2(\mathbb{D}^2).
$$

We also denote by $t_z$ the unilateral shift on the classical Hardy space $H^2(\mathbb{D})$, defined by
$$
t_z g(z)=zg(z), \qquad g\in H^2(\mathbb{D}).
$$

Let $H^2(z)$ and $H^2(w)$ denote the Hardy spaces on $\mathbb{D}$, expressed in the variables $z$ and $w$, respectively. Every function $f\in H^2(\mathbb{D}^2)$ admits a unique expansion
$$
f(z,w)=\sum_{n=0}^{\infty}f_n(z)w^n,
$$
where $f_n\in H^2(z)$ and
$$
\|f\|_{H^2(\mathbb{D}^2)}^2=\sum_{n=0}^{\infty}\|f_n\|_{H^2(z)}^2.
$$
Thus $H^2(\mathbb{D}^2)$ can be identified with the $H^{2}(z)$-valued Hardy space in the variable $w$.

For each $n\in\mathbb N_0$, set
$$
H_n=H^2(z)w^n.
$$
The preceding representation yields the orthogonal decomposition
\begin{equation}\label{eq1}
H^2(\mathbb D^2)=\bigoplus_{n=0}^{\infty}H_n.
\end{equation}
Since multiplication by $z$ leaves each $H_n$ invariant, and the same holds for its adjoint $T_z^*$, every $H_n$ is a reducing subspace for $T_z$. We denote by
$$
P_n:H^2(\mathbb D^2)\rightarrow H_n
$$
the orthogonal projection onto $H_n$. The operators $P_{n}$ provide a natural way to transfer the analysis of invariant subspaces in $H^{2}(\mathbb{D}^{2})$ to the one-variable Hardy space $H^{2}(\mathbb{D})$. In fact, we have.

\begin{proposition}\cite[Proposition 1.3]{Carmo2}
Let $M\subset H^{2}(\mathbb{D}^{2})$ be a $T^{*}_{z}$-invariant subspace. Then, for every $n\in\mathbb{N}$, $\overline{P_{n}(M)}$ is $T^{*}_{z}$-invariant subspace. Equivalently, under the identification $H_{n}\simeq H^{2}(z)$, $\overline{P_{n}(M)}$ corresponds to a $t_{z}^{*}$-invariant subspace of $H^{2}(\mathbb{D})$.
\end{proposition}

\subsection{Invariant subspaces of shift operators}

Let us remember that for each $\alpha\in\mathbb{D}$ and each $f\in H^{2}(\mathbb{D})$, the reproducing kernel function has the property
$$
\langle f,k_{\alpha}\rangle=f(\alpha).
$$
The action of $t_{\phi}^{*}$ on $k_{\alpha}$ is given by
$t^{*}_{\phi}k_{\alpha}=\overline{\phi(\alpha)}k_{\alpha},
$ which means that each $k_{\alpha}$ is an eigenvector for $t^{*}_{\phi}$ for $\phi\in H^{\infty}(\mathbb{D})$. 

Given $g\in H^2(\mathbb{D})$, we denote by $\mathcal K_g$ the smallest $t_z^*$-invariant subspace of $H^2(\mathbb{D})$ containing $g$, namely
$$
\mathcal K_g=\overline{\operatorname{span}}
\{(t_z^*)^n g:n\in\mathbb N_0\}.
$$

The following proposition establishes the existence of a nontrivial invariant subspace for the restriction of $t_z^*$ to any infinite-dimensional invariant subspace.

\begin{proposition}
If $M\subset H^2(\mathbb{D})$ is an infinite dimensional invariant subspace of $t^{*}_{z}$, then $t^{*}_{z}|_{M}$ has a non-trivial invariant subspace.
\end{proposition}
\begin{proof} If $M$ has an eigenvector $f$ of $t^{*}_{z}$, there is nothing to show because in this case $\mathcal K_f\subset M$. Suppose then that $M$ does not have eigenvectors of $t^{*}_{z}$. By Beurling's Theorem we have $M^{\perp}=\phi H^2(\mathbb{D})$ for some inner function $\phi$. Because $k_\alpha\notin M$ for all $\alpha \in \mathbb{D}$, we have that $\phi(\alpha)\neq 0$ for all $\alpha\in \mathbb{D}$ and so $\phi$ is a singular inner function. We know that a singular inner function is defined by a measure $\mu$, and we get singular inner functions $\phi_{\gamma}$, defined by measure $\gamma \mu$, such that $\phi H^2(\mathbb{D})\subset \phi_{\gamma}H^2(\mathbb{D})$ by \cite[Thm. 2.6.7]{Martinez}, as wished
\end{proof}

\section{Structural restrictions on minimal invariant subspaces}

For each $f\in H^2(\mathbb{D})$, we consider $f^{\sharp}\in H^{2}(\mathbb{D})$ given by $f^\sharp(z)=\overline{f(\overline{z})}$. Given  $g\in H^2(\mathbb{D})$, define the operator $L_g:H^{2}(\mathbb{D})\rightarrow H^{2}(\mathbb{D})$ by
$$
L_{g}f(z)=\sum_{i=0}^{\infty}\langle (t^{*}_{z})^{i}g(z),f^\sharp(z)\rangle z^{i}.
$$

Our approach is based on the universality of $T_{z}^{*}$. Given a $T_{z}^{*}$-invariant subspace $M$, we will show that there exist $n\in\mathbb{N}_{0}$ and a $T_z^*$-invariant subspace $U$ such that
$$
U\subsetneq \overline{P_n(M)}
\qquad\text{and}\qquad
P_n(M)\cap U\neq\{0\}.
$$
These properties will allow us to conclude that $M$ is not a minimal $T_{z}^{*}$-invariant subspace.

Do Carmo and Ferreira \cite{Carmo} demonstrated sufficient conditions for the ISP to hold. In our context, we have the following.

\begin{theorem}\label{t2}
If for each $g\in H^2(\mathbb{D})$, there is an invariant subspace $U\subset V_{t_{z}^{*},g}$ of $t^{*}_{z}$ such that $L_{g}(l^2(\mathbb{C}))\cap U\neq \lbrace 0\rbrace$, then the ISP is true.
\end{theorem}
\begin{proof}
It follows directly from \cite[Theorem 16]{Carmo}.
\end{proof}

In \cite{Carmo2}, do Carmo and Ferreira showed that if $M\subset H^{2}(\mathbb{D}^{2})$ is a $T_z^{*}$-invariant subspace and $g\in P_n(M)$, then $L_g(H^{2}(\mathbb{D}))\subset P_n(M)$. As a consequence, we have the following.

\begin{proposition}\label{t1}
Let $M\subset H^2(\mathbb{D})$ be a $t_z^*$-invariant subspace. If $\ker(t_z^*)\cap M=\{ 0\}$, then $t_{z}^{*}|_{M}$ is an bijection.
\end{proposition}
\begin{proof}
We must show that $t_{z}^{*}|_{M}$ is surjective, that is, that $M\subset t_{z}^{*}(M)$. Let $f\in M$. For $g(z)=\displaystyle\sum_{i=1}^{\infty}\overline{\alpha}_i z^i\in H^{2}(\mathbb{D})$, it follows from the well-definedness of $L_g$ that 
$$
\sum_{i=1}^{\infty}\alpha_it_z^{*i}f=\langle t_z^{*i}f,g\rangle z^i
$$ 
and therefore 
$$ 
\sum_{i=1}^{\infty}\alpha_it_z^{*i}f=0\iff\sum_{i=1}^{\infty}\overline{\alpha}_i z^i\in M^{\perp}. 
$$

Since $\ker(t_z^*)\cap M=\{ 0\}$ there exist $g(z)= \sum_{i=1}^{\infty}\overline{\alpha}_i z^i\in M^{\perp}$ and $ \alpha_0\neq 0$ such that
$$
\alpha_{0}f=\sum_{i=1}^{\infty}-\alpha_it_{z}^{*i}f\in t_z^*(M)
$$
and thus $f\in t^{*}_{z}(M)$.
\end{proof}

To establish the main result of this section, we need two lemmas.

\begin{lemma}\label{t3}
Let $M\subseteq H^2(\mathbb{D}^2)$ be a $T_{z}^{*}$-invariant subspace. Suppose there exists $n\in\mathbb{N}_{0}$ such that $\{0\}\neq\overline{P_n(M)}=\overline{\operatorname{span}}(U_0+U_1)$, where $U_0$ and $U_1$ are $t_{z}^{*}$-invariant subspaces and $U_0\setminus U_1$ and $U_1\setminus U_0$ are both non-empty. Under these conditions, $M$ is non-minimal.
\end{lemma}
\begin{proof}
Given that $u\in  P_n(M)$ is non-zero, there exist sequences $(x_i)$ in $U_0$ and $(y_i)$ in $U_1$, such that $(x_i+y_i)\rightarrow  u$. Thus 
$$
L_{x_i+y_i}(h) \rightarrow L_{u}(h)
$$
for all  $h\in H^2(\mathbb{D})$.

If $u\notin U_0$, there exists $h^\sharp\in U_{0}^{\perp}$ such that $L_{u}(h)\neq 0$. 
Since
\begin{align*}
L_{x_i+y_i}(h) &=\sum_{n=0}^{\infty}\langle t_{z}^{*(k)}(x_i+y_i),h^\sharp\rangle z^k \\
&= \sum_{n=0}^{\infty}\langle t_{z}^{*(k)}(y_i),h^\sharp\rangle z^k \\
&= L_{y_i}(h)\in U_1
\end{align*}
we have $L_u(H^2(\mathbb{D}))\cap U_1\neq \lbrace 0\rbrace$ and therefore $M$ is non-minimal.
\end{proof}

\begin{lemma}\label{t4}
If $V=\overline{\operatorname{span}}(U_0+U_1)$ is a $t^{*}_{z}$-invariant subspace, where $ U_{0}^{\perp}=\phi_0 H^2(\mathbb{D})$ and $U_{1}^{\perp}=\phi_1 H^2(\mathbb{D})$, then $V^{\perp}=\phi H^2(\mathbb{D})$, where $\phi$ is the least common multiple of $\phi_0$ and $\phi_1$.
\end{lemma}
\begin{proof}
Since $V$ is a $t_{z}^{*}$-invariant subspace, by Beurling's theorem, there exists an inner function $\psi$ such that $V^{\perp}=\psi H^2(\mathbb{D})$.

Note that $U_0\subsetneq V$ and $U_1\subsetneq V$. Thus $\phi_0$ and $\phi_1$ divides $\psi$ and so the least common multiple of $\phi_0$ and $\phi_1$ divides $\psi$.

Let $f\in V$ and $\zeta$ a common multiple of $\phi_0$ and $\phi_1$. Then there exist sequences $(x_n)$ in $U_0$ and $(y_n)$ in $U_1$, such that $(x_n+y_n)\rightarrow  f$. We have $\langle \zeta,x_n+y_n\rangle=0$, because $\zeta\in \phi_0 H^2(\mathbb{D})$ and $\zeta\in \phi_1 H^2(\mathbb{D})$. Therefore $\psi$ divides the smallest common multiple of $\phi_0$ and $\phi_1$, as desired.
\end{proof} 

The following theorem provides a structural restriction on infinite-dimensional minimal $T_{z}^{*}$-invariant subspaces of $H^2(\mathbb{D}^2)$, showing that the closure of each projected subspace must satisfy one of two specific alternatives.

\begin{theorem}\label{t5}
If $M\subset H^2(\mathbb{D}^2)$ is an infinite-dimensional minimal $T_z^*$-invariant subspace, then for every $n\in\mathbb{N}_{0}$, either 
$$
\overline{P_n(M)}=H^2(\mathbb D) \ \text{or} \ \overline{P_n(M)}^\perp
=\phi_{a,\theta}H^2(\mathbb D),
$$
where
$$
\phi_{a,\theta}(z)=\exp\left(a\frac{z+e^{i\theta}}{z-e^{i\theta}}\right),
\qquad a>0,\quad \theta\in\mathbb{R}.
$$
\end{theorem}
\begin{proof}
For each $n\in\mathbb{N}_{0}$, since $V_{n}=\overline{P_n(M)}$ is a $t^*_z$-invariant subspace we have by Beurling’s Theorem that $V_{n}^{\perp}=B\phi H^2(\mathbb{D})$, where $B=B(n)$ is a Blaschke product and $\phi=\phi(n)$ is an inner function. Thus
$$
\overline{P_n(M)}=H^2(\mathbb{D}) \ \text{or} \ \overline{P_n(M)}^{\perp}=B\phi H^2(\mathbb{D}).
$$

We must then show that $B$ is constant and $\phi=\phi_{a,\theta}$ where $\phi_{a,\theta}(z)=\exp\left(a\frac{z+e^{i\theta}}{z-e^{i\theta}}\right)$ for some $a>0$ and $\theta\in\mathbb{R}$.

If $B$ and $\phi_{a,\theta}$ are non constants then 
$$
V_{n}=\overline{\operatorname{span}}(BH^2(\mathbb{D})^{\perp}+\phi_{a,\theta} H^2(\mathbb{D})^{\perp})
$$ 
therefore by Lemma \ref{t3} we have $M$ is non minimal.

Suppose then that $\phi_{a,\theta}$ is constant, that is, $V_{n}^{\perp}=BH^2(\mathbb{D})$. If $B$ has just only one zero in $H^2(\mathbb{D})$ we have that $V_{n}$ is finite dimensional and so $M$ is non minimal. On the other hand, if $B$ has two or more zeros in $H^2(\mathbb{D})$ then $B=B_1 B_2$, with $B_1$ and $B_2$ non-constant and cooprimes. Thus by Lemmas \ref{t3} and \ref{t4}, we have $M$ is non minimal.

Having established that $B$ is constant, assume by contradiction that $\phi_{a,\theta}(z)\neq \exp\left(a\frac{z+e^{i\theta}}{z-e^{i\theta}}\right)$ for all $z\in\mathbb{D}$ and $a>0$. Since $\phi_{a,\theta}$ is a singular inner function, by the Herglotz representation theorem there exists a unique positive singular finite Borel measure $\mu$ on $[0, 2\pi)$ such that
$$
\phi_{a,\theta}(z) = \exp \left( -\int_{0}^{2\pi} \frac{e^{i\theta} + z}{e^{i\theta} - z} \, d\mu(\theta) \right), \quad z \in \mathbb{D}.
$$

If $E_0, E_1 \subset [0, 2\pi)$ are disjoint Borel sets with $\mu(E_0) \neq 0$ and $\mu(E_1) \neq 0$, we decompose $\phi_{a,\theta}=\phi_0\phi_1$, where
$$
\phi_k(z) =\exp\left(-\int_{0}^{2\pi}\frac{e^{i\theta}+z}{e^{i\theta}-z}\,d\mu_k(\theta)\right),\quad k\in\{0, 1\},
$$
and $\mu_k = \mu|_{E_k}$, i.e., $\mu_k(E) = \mu(E \cap E_k)$ for every Borel set $E \subset [0, 2\pi)$. Thus considering $U_0^{\perp}=\phi_0 H^2(\mathbb{D})$ and $U_1^{\perp}=\phi_1 H^2(\mathbb{D})$, we have by Lemmas \ref{t3} and \ref{t4} that $M$ is non minimal. Therefore $\phi_{a,\theta}(z)=\exp\left(a\frac{z+e^{i\theta}}{z-e^{i\theta}}\right)$ for some $a>0$ and $\theta\in\mathbb{R}$ and we have the desired result.
\end{proof}


Lemmas \ref{t3} and \ref{t4} establish the non-minimality of $M$ whenever $\overline{P_n(M)}=K_\phi$ and $\phi$ admits an inner factorization $\phi=vw$ with $v\nmid w$ and $w\nmid v$. However, this approach does not apply to atomic singular inner functions, whose inner divisors are always comparable. This remaining case is addressed in the next section.

\section{Density of projections of minimal invariant subspaces}

In this section, we develop an alternative approach based on inner factorizations and cyclicity. This allows us to rule out the second alternative
in Theorem \ref{t5} and establish our main result: every nonzero coordinate projection of an infinite-dimensional minimal $T_z^*$-invariant subspace is dense in $H^2(\mathbb{D})$.


\begin{lemma}
Let $P:L^2(\mathbb{T})\to H^2(\mathbb{D})$ be the
orthogonal projection. If $g\in H^2(\mathbb{D})$,
then
$$
L_g(f)=P(\overline{f^\sharp}g),
\qquad f\in H^\infty(\mathbb{D}).
$$
\end{lemma}
\begin{proof}
The identity follows directly from the definition of $L_g$ for polynomials. The general case follows by radial approximation, since $f\in H^\infty(\mathbb{D})$.
\end{proof}


\begin{theorem}\label{t7}
Let $u\in H^\infty(\mathbb{D})$ be an inner function
and let $g\in K_u$ be a cyclic vector for $t_z^*|_{K_u}$.
If $u=\vartheta w$, where $\vartheta$ and $w$ are
nonconstant inner functions, then
$$
L_g(H^\infty(\mathbb{D}))\cap K_\vartheta
\neq\{0\}.
$$
\end{theorem}
\begin{proof}
Since $g\in K_u$, we have
$g=\overline{zh}\,u$ for some $h\in H^2(\mathbb{D})$.
Taking $f^\sharp=w$, we obtain
$$
L_g(f)=P(\overline{w}g)
=P(\overline{zh}\,\vartheta)\in K_\vartheta.
$$

Moreover, by the cyclicity of $g$ in $K_u$, we have
$$
\ker(L_g|_{H^\infty(\mathbb{D})})=
\{f\in H^\infty(\mathbb{D}):f^\sharp\in uH^2(\mathbb{D})\}.
$$
Since $\vartheta$ is nonconstant, we have $w\notin uH^2(\mathbb{D})$. Therefore, $L_g(f)\neq 0$, which completes the proof.
\end{proof}


Unlike the approach based on Lemmas \ref{t3} and \ref{t4},
Theorem \ref{t7} does not require the inner factors of $u$
to be incomparable. In particular, every nonconstant
singular inner function $\phi$ admits the factorization
$$
\phi=\phi^\alpha\phi^{1-\alpha},
\qquad 0<\alpha<1.
$$
This observation, together with the cyclicity arising from minimality, enables us to exclude the remaining alternative in Theorem \ref{t5} and establish our main result.

Note that every closed $T_z^*$-invariant subspace is also invariant under $T_\psi^*$ for every $\psi\in H^\infty(\mathbb{D})$, where $T_\psi$ denotes multiplication by $\psi(z)$ on $H^2(\mathbb{D}^2)$. This follows by radial approximation of $\psi$ by polynomials.

\begin{theorem}\label{t6}
Let $M\subset H^2(\mathbb{D}^2)$ be an infinite-dimensional
minimal $T_z^*$-invariant subspace. Then, for every
$n\in\mathbb{N}_0$, either
$$
P_n(M)=\{0\}
\quad\text{or}\quad
\overline{P_n(M)}=H^2(\mathbb{D}).
$$
\end{theorem}
\begin{proof}
Suppose, by contradiction, that $\{0\}\neq\overline{P_n(M)}\neq H^2(\mathbb{D})$. By Theorem \ref{t5}, we have $\overline{P_n(M)}=K_u$, where $u=\phi_{a,\theta}$.

Choose $0\neq g\in P_n(M)$ and $G\in M$ such that $P_n(G)=g$. By the minimality of $M$ and the intertwining relation between $P_n$ and $T_z^*$, we obtain
$$
\overline{\operatorname{span}}
\{(t_z^*)^j g:j\in\mathbb{N}_0\}=K_u.
$$

Since $u$ is a singular inner function, we may write $u=\vartheta\psi$, where $\vartheta=u^\alpha$ and $\psi=u^{1-\alpha}$, for some $0<\alpha<1$.

By Theorem \ref{t7}, taking $f^\sharp=\psi$, we obtain
$$
0\neq L_g(f)=t_\psi^*g\in K_\vartheta.
$$

Since $T_\psi^*$ leaves $M$ invariant and
$P_n(T_\psi^*G)=t_\psi^*g$, it follows that
$$
P_n(M)\cap K_\vartheta\neq\{0\}.
$$

Now consider the closed $T_z^*$-invariant subspace
$$
N=\{F\in M:P_n(F)\in K_\vartheta\}.
$$

The preceding intersection implies that $N\neq\{0\}$. Moreover, since
$K_\vartheta\subsetneq K_u=\overline{P_n(M)}$, we have $N\neq M$, contradicting the minimality of $M$.
\end{proof}

\begin{question}
Let $M$ be an infinite-dimensional closed $T_z^*$-invariant subspace of $H^2(\mathbb D^2)$ such that, for every $n\in\mathbb N_0$,
$$
P_n(M)=\{0\}
\quad\text{or}\quad
\overline{P_n(M)}=H^2(\mathbb D).
$$
Must $M$ be non-minimal?
\end{question}

An affirmative answer to this question, together with Theorem \ref{t6} and the universality of $T_z^*$, would yield a positive solution to the Invariant Subspace Problem.

\subsection*{Declaration of AI assistance}
The authors used ChatGPT (OpenAI) to assist in writing the manuscript.

\subsection*{Ethics declarations }
The authors declare that they have no conflict of interest.

\subsection*{Data availability statement}
Data sharing not applicable to this article as no datasets were generated or analysed during the current study.

\subsection*{Funding}
This work did not receive any external funding.

\subsection*{Acknowledgment}
The second author gratefully acknowledges the institutional support and infrastructure provided by the Universidade Estadual de Santa Cruz during the development of this work.

\end{document}